\documentclass{article}

\usepackage[all]{xy}
\usepackage{mathrsfs}
\usepackage{amssymb}
\usepackage{amsthm}
\usepackage{amsmath}

\usepackage{tikz}
\usetikzlibrary{arrows,positioning}

\newtheorem{thm}{Theorem}

\newtheorem{prop}[thm]{Proposition}

\author{Camilo Sanabria Malag\'on\thanks{Partially supported by Vicerrector\'ia de
Investigaciones de la Universidad de los Andes grant PEP P13.160422.030 FAPA – Camilo Sanabria.}\\Department of Mathematics\\ Universidad de los Andes, Bogot\'a, Colombia}

\title{An algorithm for computing differential equations for invariant curves}

\begin{document}

\date{}

\maketitle

\begin{abstract}
In this paper we describe an algorithm based on the Picard-Vessiot theory that constructs, given any curve invariant under a finite linear algebraic group over the complex numbers, an ordinary linear differential equation whose Schwarz map parametrizes it.
\end{abstract}

\section*{Introduction}

A classical result of Hurwitz \cite{HURWITZ1892} states that a curve of genus $g\ge 2$ can have at most $84(g-1)$ automorphisms. In \cite{KLEIN1878}, Klein described the curve of genus $3$ with $168$ automorphisms, proving that Hurwitz's bound is optimal. This curve is the locus of $$x^3y+y^3z+z^3x=0$$ in $\mathbb{P}^2(\mathbb{C})$ and it's known as Klein quartic. Halphen \cite{HALPHEN1884} and Hurwitz \cite{HURWITZ1886} independently obtained linear ordinary differential equations such that the image of their Schwarz maps is Klein quartic and used them to parametrize the curve.

The automorphism group of Klein quartic is the famous Klein's simple group of order $168$, $G_{168}$. The curves invariant under this group have been widely studied \cite{ADLER1999,FRICKE1893,HIRZEBRUCH1977}. In particular, Fricke in \cite{FRICKE1893} described a pencil of invariant projective curves of degree 12. For each curve in this pencil Kato constructed an ordinary linear differential equation of order 3 such that the image of its Schwarz map parametrize the given curve \cite{KATO2004}. To construct the equations Kato used some properties of the invariant curves, in particular their cusps and flexes. Furthermore, in \cite{KATO2006}  Kato proved that the Schwarz map of a linear ordinary differential equation obtained by Beukers in \cite{BEUKERS1989} parametrizes the Hessian of Klein quartic.

More generally, curves invariant under other finite linear algebraic groups have been studied. For example, the Hesse pencil \cite{HESSE1844I, HESSE1844II}, a family of elliptic curves invariant under the normal abelian subgroup of order $9$ of the Hessian group has gained much attention recently in different areas \cite{ARTEBANI2009,SMART2001,ZASLOW2005}. Lachtin in \cite{LACHTIN1898} considered curves invariant under Valentiner group, the unique triple cover of $A_6$.

The main result of this paper is Theorem \ref{theo} which is a generalization of Theorem 6 in \cite{SANABRIA2017}. From its proof we obtain the algorithm that constructs, given any curve invariant under a finite linear algebraic group over $\mathbb{C}$, an ordinary linear differential equation such that its Schwarz map parametrizes this curve. To apply this algorithm we do not use any property of the invariant curves, but instead their properties can be derived from their corresponding differential equations. In the last section we revisit the known examples of parametrizations of curves invariant under $G_{168}$ showing the information needed so that we can obtain those parametrizations using our algorithm.

A MAPLE implementation of the algorithm for third order equations can be obtained from my webpage.

\section{Preliminary}

\subsection{Space of orbits}

Let $G\subseteq GL_n(\mathbb{C})$ be a finite linear algebraic group. We define a left group action of $G$ on the coordinate ring $R=\mathbb{C}[X_1,\ldots,X_n]$ of $\mathbb{C}^n$ by $\mathbb{C}$-automorphisms
\[
X_i\mapsto \sum_{j=1}^n g_{ij}X_j,
\]
for $(g_{ij})_{i,j=1}^n\in G$.
This $G$-action on $R$ induces a right $G$-action on $\mathbb{C}^n$ defined by
\[
(x_1,\ldots,x_n)\mapsto \Big(\sum_{i=1}^ng_{ij}x_i,\ldots,\sum_{i=1}^ng_{in}x_i\Big).
\]
Since $G$ is finite, it is reductive, thus the $G$-invariant polynomials in $R$ separate the $G$-orbits in $\mathbb{C}^n$. Therefore, the coordinate ring of the orbit space $\mathbb{C}^n/G$ is the finitely generated $G$-invariant subring $R^G$.

Let $F_1,\ldots,F_N$ be generators of $R^G$. We can embed $\mathbb{C}^n/G$ into $\mathbb{C}^N$ through the algebraic map
\[
\mathbf{x}\cdot G \mapsto \Big(F_1(\mathbf{x}),\ldots,F_N(\mathbf{x})\Big),
\]
where $\mathbf{x}=(x_1,\ldots,x_n)$. We will identify the orbits space $G\backslash\mathbb{C}^n$ with its embedding in $\mathbb{C}^N$. Since $G$ is finite, the orbit space has dimension $n$ and therefore there exits a dense subset of $\mathbb{C}^n$ where the derivative of the quotient map
\[
\mathbf{x}\mapsto \Big(F_1(\mathbf{x}),\ldots,F_N(\mathbf{x})\Big)
\]
is non-singular.

Similarly, these $G$-actions define $G$-actions on $\mathbb{P}^{n-1}(\mathbb{C})$ and on its homogeneous coordinate ring $R$. The orbits space of the action on $\mathbb{P}^{n-1}(\mathbb{C})$ is $\mathbb{P}(\mathbb{C}^n/G)$. Let $\Lambda\in\mathbb{Z}_{>0}$ be such that the homogeneous elements of $R^G$ of degree $\Lambda$, $R^G_\Lambda$, form a homogeneous coordinate system for $\mathbb{P}(\mathbb{C}^n/G)$, i.e. $$\mathbb{P}(R^G_\Lambda)\simeq\mathbb{P}(\mathbb{C}^n/G).$$

Let $\Phi_1,\ldots,\Phi_M$ be a basis of $R^G_\Lambda$. We can embed $\mathbb{P}(\mathbb{C}/G)$ into $\mathbb{P}^{M-1}(\mathbb{C})$ with the map
\[
[\mathbf{x}]\cdot G\mapsto \Big[\Phi_1(\mathbf{x}):\ldots:\Phi_M(\mathbf{x})\Big]
\]
where $[\mathbf{x}]=[x_1:\ldots:x_n]$. As before, we identify the orbits space $\mathbb{P}(\mathbb{C}^n/G)$ with its embedding in $\mathbb{P}^{M-1}(\mathbb{C})$. Again, there exits a dense subset of $\mathbb{P}^{n-1}(\mathbb{C})$ where the derivative of the quotient map
\[
[\mathbf{x}]\mapsto \Big[\Phi_1(\mathbf{x}):\ldots:\Phi_M(\mathbf{x})\Big]
\]
is non-singular.

\subsection{Schwarz maps}

Let $C_0$ be a compact Riemann surface and $K=\mathbb{C}(C_0)$ the field of meromorphic functions over $C_0$. Let $\delta:K\rightarrow K$ be any non-trivial derivation. Note that $\delta$ can be uniquely extended to the sheaf of meromorphic functions over any open set in $C_0$.

Let $L(y)=0$ be an ordinary linear differential equation of order $n$ with rational coefficients
\[
L(y)=\delta^n(y)+a_{n-1}\delta^{n-1}(y)+\ldots+a_1\delta(y)+a_0y,
\]
where $a_0,a_1,\ldots,a_{n-1}\in K$. Let $S\subset C_0$ be the collection of singularities of $L(y)$. Given a non-singular point $p\in C_0$ together with a fundamental system of solutions $(y_1,\ldots,y_n)$ over a neighborhood $U\subseteq C_0$ of $p$, we define the Schwarz map as the analytic extension of 
\begin{eqnarray*}
 U & \longrightarrow & \mathbb{P}^{n-1}(\mathbb{C})\\
z & \longmapsto & \Big[y_1(z):\ldots:y_n(z)\Big].
\end{eqnarray*}

The monodromy of the Schwarz map is the projection of the monodromy group of the fundamental system of solutions $(y_1,\ldots,y_n)$ on $PGL_n(\mathbb{C})$. We will denote it by $G_0$ and call it the projective monodromy of $L(y)=0$. Post-composing the Schwarz map with the quotient by the action of $G_0$ on  $\mathbb{P}^{n-1}(\mathbb{C})$ we obtain a single-valued map $$\psi:C_0\setminus S\rightarrow\mathbb{P}^{n-1}(\mathbb{C})/G_0$$ which we will call the quotient Schwarz map.

\section{Schwarz maps for invariant curves}

\begin{thm}\label{theo}
Let $G$ be a finite algebraic subgroup of $GL_n(\mathbb{C})$ and let $C\subseteq\mathbb{P}^{n-1}(\mathbb{C})$ be an algebraic $G$-invariant curve not contained in a projective hyperplane. Let $C_0$ be a compact Riemann surface  such that there is a dominant map $$\psi: C_0\setminus S\rightarrow C/G\subseteq\mathbb{P}(\mathbb{C}^n/G)$$ with $S\subseteq C_0$ finite and such that $C\rightarrow C/G$ is unramified over $\psi(C_0\setminus S)$. Let $K=\mathbb{C}(C_0)$ be the field of meromorphic functions over $C_0$. Then there exist
\begin{itemize}
\item[i)] a branched cover $\pi:C_1\rightarrow C_0$ where $C_1$ is a compact Riemann surface and $K\subseteq E=\mathbb{C}(C_1)$ is an abelian extension;
\item[ii)] a linear differential equation $L(y)=0$ of order $n$ with coefficients in $E$ admitting a fundamental system of solutions $(y_1(z),\ldots,y_n(z))$ over $C_1$ such that the closure of the image of its Schwarz map is $C$.
\end{itemize}
Moreover, the restriction of $\psi\circ\pi$ to $C_1\setminus\pi^{-1}(S)$ is the quotient Schwarz map of $L(y)=0$ associated to the system of solutions $(y_1(z),\ldots,y_n(z))$. 
\end{thm}

\begin{proof}
Let $F_1,\ldots,F_N$ be generators of $R^G$. Let $\Lambda\in\mathbb{Z}_{>0}$ be such that the homogeneous elements of $R^G$ of degree $\Lambda$, $R^G_\Lambda$, form a homogeneous coordinate system for $\mathbb{P}(\mathbb{C}^n/G)$. Let $\Phi_1,\ldots,\Phi_M$ be a basis of $R^G_\Lambda$. Let $\delta:K\rightarrow K$ be a non-trivial derivation.

Let $d\in\{1,\ldots,M\}$ be such that $C\not\subseteq\{\Phi_d=0\}$. Then the map $\psi$ is completely determined by the pullback functions $\psi^*\left(\Phi_i/\Phi_d\right)$, $i=1,\ldots,M$.

Let $\phi_1,\ldots,\phi_M\in K$ be such that 
\begin{eqnarray*}
\frac{\phi_i}{\phi_d} & = & \psi^*\left(\frac{\Phi_i}{\Phi_d}\right)
\end{eqnarray*}
for $i=1,\ldots,M$, and let $f_1,\ldots,f_{N}\in \overline{K}$ be such that
\begin{eqnarray}
\prod_{i=1}^Nf_i^{n_i} & = & \phi_j\label{Ff}
\end{eqnarray}
whenever $\prod_{i=1}^NF_i^{n_i}=\Phi_j$, for $j=1,\ldots,M$. Hence the extension $K\subseteq K(f_1,\ldots,f_N)$ is abelian.

Let $E=K(f_1,\ldots,f_N)$. The unique extension of $\delta$ to a derivation in $E$ will also be denoted by $\delta$. Let $C_1$ be a compact Riemann surface such that $\mathbb{C}(C_1)\simeq E$ and let $\pi:C_1\rightarrow C_0$ be the map defined by the extension $K\subseteq E$. Then the map $\pi\circ\psi$ is the composition of two algebraic maps, $\Pi\circ\Psi$, where
$$\Psi:C_1\setminus \pi^{-1}(S)\rightarrow(\mathbb{C}^n\setminus\{\overline{0}\})/G$$
is such that $f_i=\Psi^*(F_i)$, for $i=1,\ldots,N$, and $\Pi$ is the canonical projection
$$\Pi:(\mathbb{C}^n\setminus\{\overline{0}\})/G\rightarrow\mathbb{P}(\mathbb{C}^n/G).$$

Let $p\in C_1\setminus \pi^{-1}(S)$ and let us consider a lifting $\widetilde{\Psi}: U\rightarrow\mathbb{C}^n\setminus\{\overline{0}\}$ of $\Psi$ over a simply connected neighborhood $U\subseteq C_1\setminus \pi^{-1}(S)$ of $p$
\[
\xymatrix{
 &  & (y_1(z),\ldots,y_n(z))\ar@{|->}[d] & \mathbb{C}^n\setminus\{\overline{0}\}\ar[d] \\
z\ar@{|->}[urr]^{\widetilde{\Psi}}\ar@{|-->}[rr]_{\Psi} & & (y_1(z),\ldots,y_n(z))\cdot G & (\mathbb{C}^n\setminus\{\overline{0}\})/G
}
\]
If we differentiate  with $\delta$ the equations $$f_i=\Psi^*(F_i),\quad i=1,\ldots,N$$ we obtain
\[
\frac{\partial F_i}{\partial X_1}\delta(y_1)+\ldots+\frac{\partial F_i}{\partial X_n}\delta(y_n)=  \delta(f_i),\quad i=1,\ldots,N.
\]
After arranging the indices of the $F_i$'s if necessary, we have that
$$\frac{\partial(F_1,\ldots,F_n)}{\partial(X_1,\ldots,X_n)}(\mathbf{x})=\left[\begin{array}{ccc}
\partial P_1/\partial X_1 & \ldots & \partial P_1/\partial X_n\\
\vdots & \ddots & \vdots\\
\partial P_n/\partial X_1 & \ldots & \partial P_n/\partial X_n
\end{array}\right](\mathbf{x})$$
is invertible over $U$. If $\det\Big(\frac{\partial(F_1,\ldots,F_n)}{\partial(X_1,\ldots,X_n)}\Big)(\mathbf{y})\ne 0$, where $\mathbf{y}=(y_1,\ldots,y_n)$, then we obtain a system of first order ordinary differential equations characterizing the functions $y_1(z),\ldots ,y_n(z)$
\[
\delta\left[\begin{array}{c}
y_1 \\ \vdots \\ y_n
\end{array}\right] =
\left[\frac{\partial(F_1,\ldots,F_n)}{\partial(X_1,\ldots,X_n)}(\mathbf{y})\right]^{-1}\ \delta\left[\begin{array}{c}
f_1 \\ \vdots \\ f_n
\end{array}\right].
\] 
Differentiating this system $n-1$ times with $\delta$, we obtain expressions for
$$\delta^k\left[\begin{array}{c}
y_1 \\ \vdots \\ y_n
\end{array}\right],\quad k=2,\ldots,n$$
in terms of the derivatives of $f_1,\ldots, f_n$, and the partial derivatives of $F_1,\ldots,F_n$. The linear dependence relations
$$c_0\left[\begin{array}{c}
y_1 \\ \vdots \\ y_n
\end{array}\right]+c_1
\delta\left[\begin{array}{c}
y_1 \\ \vdots \\ y_n
\end{array}\right]+\quad
\ldots\quad+
c_{n-1}\delta^{n-1}\left[\begin{array}{c}
y_1 \\ \vdots \\ y_n
\end{array}\right]+
c_n\delta^n\left[\begin{array}{c}
y_1 \\ \vdots \\ y_n
\end{array}\right]=0$$
are invariant under $G$, hence $$c_m\in\mathbb{C}\left(F_i(\mathbf{y}),\delta^k f_j\Big|\ i,j,k=1,\ldots,n\right),\quad m=0,\ldots,n.$$ Furthermore, since $F_i(\mathbf{y})=\Psi^*(F_i)=f_i$, for $i=1,\ldots,N$, we have $$c_m\in\mathbb{C}\left(\delta^kf_j\Big|\ j=1,\ldots,N,\ k=0,1,\ldots,n\right)\subseteq E,\quad m=0,\ldots,n.$$
Therefore, the linear differential equation $$c_n\delta^n(y)+c_{n-1}\delta^{n-1}(y)+\ldots+c_1\delta(y)+c_0y=0$$ has fundamental system of solutions $(y_1(z),\ldots,y_n(z))$ whose analytic extension parametrizes a $G$-invariant curve in $\mathbb{C}^n\setminus\{(0,0,0)\}$. Passing to homogeneous coordinates, the analytic extension of $[y_1(z):\ldots:y_n(z)]$ defines a Schwarz map whose image is a $G$-invariant curve in $\mathbb{P}^{n-1}(\mathbb{C})$ whose closure is $C$. By construction the quotient Schwarz map is $\psi\circ\pi$.
\end{proof}

\subsection*{The algorithm}

Here we present the algorithm that computes the equation $L(y)=0$ from Theorem \ref{theo}.\newline
\textbf{Input}: $F_1,\ldots,F_N$ generators of $R^G$ such that $\det\Big(\frac{\partial(F_1,\ldots,F_n)}{\partial(X_1,\ldots,X_n)}\Big)\ne 0$, and $f_1,\ldots,f_N\in\overline{K}$ as in (\ref{Ff}).\newline
\textbf{Output}: $c_{n-1},\ldots,c_1,c_0\in K(f_1,\ldots,f_N)$ such that the linear ordinary differential equation $\delta^n(y)+c_{n-1}\delta^{n-1}(y)+\ldots+c_1\delta(y)+c_0y=0$ has a fundamental system of solutions whose associated Schwarz map parametrizes $C$.
\begin{itemize}
\item[(i)] Let $X_1,\ldots,X_n$ be variables. Extend the derivative $\delta$ to $K(X_1,\ldots,X_n)$ by defining $\delta X_i=0$, for $i=1,\ldots,n$.
\item[(ii)] Define recursively for $j=1,\ldots,n$
\begin{eqnarray*}
\left[\begin{array}{c}
X_{1,0} \\ \vdots \\ X_{n,0}
\end{array}\right] & = & \left[\begin{array}{c}
X_{1} \\ \vdots \\ X_{n}
\end{array}\right],\\
\left[\begin{array}{c}
X_{1,j} \\ \vdots \\ X_{n,j}
\end{array}\right]  & = & \frac{\partial(X_{1,j-1},\ldots,X_{n,j-1})}{\partial(X_1,\ldots,X_n)}\Big[\frac{\partial(F_1,\ldots,F_n)}{\partial(X_1,\ldots,X_n)}\Big]^{-1}\delta\left[\begin{array}{c}
y_1 \\ \vdots \\ y_n
\end{array}\right]\\
 & & \qquad +\quad \left[\begin{array}{c}
\delta X_{1,j-1} \\ \vdots \\ \delta X_{n,j-1}
\end{array}\right].
\end{eqnarray*}
\item[(iii)] Solve for $C_0,C_1,\ldots,C_n\in K(f_1,\ldots,f_N)(X_1,\ldots,X_n)$ the system
$$\left[\begin{array}{cccc}
X_{1,0} & X_{1,1} & \cdots & X_{1,n}\\
\vdots & \vdots & \ddots & \vdots \\
X_{n,0} & X_{n,1} & \cdots & X_{n,n}\\
\end{array}\right] \left[\begin{array}{c}
C_0 \\ \vdots \\ C_n
\end{array}\right] = \left[\begin{array}{c}
0 \\ \vdots \\ 0
\end{array}\right].$$
\item[(iv)] For $i=0,\ldots,n-1$, compute $Q_i(F_1,\ldots,F_N)\in K(f_1,\ldots,f_N)(F_1,\ldots,F_N)$ such that $$Q_i(F_1,\ldots,F_N)=\frac{C_i}{C_n}.$$
\item[(v)] \textbf{Return}: $c_i=Q_i(f_1,\ldots,f_n)$ for $i=0,\ldots,n-1$.
\end{itemize}

\section{Properties of the images of Schwarz maps}

Let $G$ be a finite algebraic subgroup of $GL_n(\mathbb{C})$. Let $C\subseteq\mathbb{P}^{n-1}(\mathbb{C})$ be an algebraic $G$-invariant curve as in Theorem \ref{theo}. We showed that we can obtain a linear ordinary differential equation $L(y)=0$ admitting a system of solutions whose associated Schwarz map parametrizes $C$.

The following proposition summarizes the properties of $C$ that we can derive from its associated equation $L(y)=0$. Note that assertion i) is Proposition 3.1 in \cite{KATO2006}, iii) is Proposition 18 in \cite{SANABRIA2014}, and  ii) and iv) follow from generalizations of the proofs of Proposition 3.1 and Theorem 3.6 in \cite{KATO2006} respectively.

\begin{prop}\label{prop}
Let $p\in\pi^{-1}(S)$ and let $$\{e_p,e_p+\frac{\nu_p}{r_p},e_p+\frac{\nu_p+\lambda_{1,p}}{r_p},\ldots,e_p+\frac{\nu_p+\lambda_{n-2,p}}{r_p}\}$$
be the characteristic exponents of $L(y)$  at $p$, where $r_p$, $\nu_p$, $\lambda_{1,p},\ldots,\lambda_{n-2,p}$ are positive integers such that $(r_p,\nu_p,\lambda_{1,p},\ldots,\lambda_{n-2,p})=1$. Let $P\in C$ be a point above $\psi\circ\pi(p)$, i.e. $P\cdot G=\psi\circ\pi(p)$.
\begin{itemize}
\item[i)] In the case $\deg(\psi\circ\pi)=1$ and $n=3$ we have that $P$ is a $(\nu_p,\nu_p+\lambda_{1,p})$-cusp if $\nu_p\ge 2$ and a $(1,1+\lambda_{1,p})$-flex if $\nu_p=1$ and $\lambda_{1,p}\ge 2$.
\item[ii)] The curve $C$ is smooth at $P$ if and only if $\nu_p=\lambda_{1,p}=\ldots=\lambda_{n-2,p}=1$ and $p$ is an apparent singularity if and only if $r_p=1$.
\item[iii)] The Euler characteristic of $C$ is
$$\chi=|G|\left(\sum_{p\in S}\left(\frac{1}{r_p}-1\right)-2\right).$$
\item[iv)] Let $m=\deg(\psi\circ\pi)$ and let $g$ be the genus of $C$. The degree of $C$ is
$$n=-\frac{g}{m}\sum_{p\in S}e_p.$$
\end{itemize}
\end{prop}

\section{Examples}

\subsection*{Invariants of Klein's group}
Klein's simple group of order $168$, $G_{168}$, is isomorphic to the subgroup of $SL_3(\mathbb{C})$ generated by the matrices
\[
\left[\begin{array}{ccc}
\beta & 0 & 0\\
0 & \beta^2 & 0\\
0 & 0 & \beta^4
\end{array}\right],\quad
\left[\begin{array}{ccc}
0 & 1 & 0\\
0 & 0 & 1\\
1 & 0  & 0
\end{array}\right],\textrm{ and }
\left[\begin{array}{ccc}
a  & b & c\\
b & c & a\\
c & a  & b
\end{array}\right],
\]
where $\beta$ is a primitive $7$-th root of unity (i.e. $\beta^6+\beta^5+\beta^4+\beta^3+\beta^2+\beta^1=0$), $a=\beta^4-\beta^3$, $b=\beta^2-\beta^5$, and $c=\beta-\beta^6$. Note that the image of $G_{168}$ under the projection $GL_3(\mathbb{C})\rightarrow PGL_3(\mathbb{C})$ is isomorphic to $G_{168}$. 

The invariant subring $\mathbb{C}[X_1,X_2,X_3]^{G_{168}}$ is generated by
\begin{eqnarray*}
F_4 & = & X_1^3X_2+X_2^3X_3+X_3^3X_1,\\
F_6 & = & \frac{1}{54}\det\left[\begin{array}{ccc}
\partial^2 F_4/\partial X_1\partial X_1 & \partial^2 F_4/\partial X_1\partial X_2 & \partial^2 F_4/\partial X_1\partial X_3 \\
\partial^2 F_4/\partial X_2\partial X_1 & \partial^2 F_4/\partial X_2\partial X_2 & \partial^2 F_4/\partial X_2\partial X_3 \\
\partial^2 F_4/\partial X_3\partial X_1 & \partial^2 F_4/\partial X_3\partial X_2 & \partial^2 F_4/\partial X_3\partial X_3 \\
\end{array}\right],\\
F_{14} & =  & \frac{1}{9}\det\left[\begin{array}{cccc}
\partial^2 F_4/\partial X_1\partial X_1 & \partial^2 F_4/\partial X_1\partial X_2 & \partial^2 F_4/\partial X_1\partial X_3 & \partial F_6/\partial X_1\\
\partial^2 F_4/\partial X_2\partial X_1 & \partial^2 F_4/\partial X_2\partial X_2 & \partial^2 F_4/\partial X_2\partial X_3 & \partial F_6/\partial X_2\\
\partial^2 F_4/\partial X_3\partial X_1 & \partial^2 F_4/\partial X_3\partial X_2 & \partial^2 F_4/\partial X_3\partial X_3 & \partial F_6/\partial X_3\\
\partial F_6/\partial X_1 & \partial F_6/\partial X_2 & \partial F_6/\partial X_3 & 0
\end{array}\right],
\end{eqnarray*}
and
\begin{eqnarray*}
F_{21} & = & \frac{1}{14}\det\left[\begin{array}{ccc}
\partial F_4/\partial X_1 & \partial F_4/\partial X_2 & \partial F_4/\partial X_3 \\
\partial F_6/\partial X_1 & \partial F_6/\partial X_2 & \partial F_6/\partial X_3 \\
\partial F_{14}/\partial X_1 & \partial F_{14}/\partial X_2 & \partial F_{14}/\partial X_3
\end{array}\right].
\end{eqnarray*}
As a ring, $\mathbb{C}[X_1,X_2,X_3]^{G_{168}}$ is isomorphic to $\mathbb{C}[\Phi_4,\Phi_6,\Phi_{14},\Phi_{21}]/(T)$ where
\begin{eqnarray*}
T & = & 2048\Phi_4^9\Phi_6 -22016\Phi_4^6\Phi_6^3 -256\Phi_{14}\Phi_4^7 +60032\Phi_4^3\Phi_6^5 +1088\Phi_{14}\Phi_4^4\Phi_6^2\\ & & -1728\Phi_6^7 +1008\Phi_{14}\Phi_4\Phi_6^4 +88\Phi_{14}^2\Phi_4^2\Phi_6 +\Phi_{14}^3-\Phi_{21}^2.
\end{eqnarray*}

\subsection{Differential equations for $G_{168}$}

In this section we consider the known examples of parametrizations of curves invariant under $G_{168}$. We show the input required by our algorithm in order to obtain the ordinary linear differential equations associated to these parametrizations. The corresponding functions defining the quotient Schwarz map $f_4$, $f_6$ and $f_{14}$ were obtained using the algorithm in \cite{VANHOEIJ1997}. 

\subsubsection*{Klein quartic}

Input: $F_4$, $F_6$, $F_{14}$, $f_4(z)=0$, $f_6(z)=\dfrac{1}{z^4}$, $f_{14}(z)=-\dfrac{12}{z^9}$\newline
Output: Hurwitz equation \cite{HURWITZ1886}
\begin{eqnarray*}
0 & = & \left(\frac{d}{dz}\right)^3y+\frac{7z-4}{z(z-1)}\left(\frac{d}{dz}\right)^2y+\frac{1}{252}\frac{2592z^2-2963z+560}{z^2(z-1)^2}\left(\frac{d}{dz}\right)y \\
 & & \ +\frac{1}{24696}\frac{57024z-40805}{z^2(z-1)^2}y
\end{eqnarray*}

\subsubsection*{Hessian of Klein quartic}

Input: $F_4$, $F_6$, $F_{14}$, $f_4(z)=\dfrac{1}{z}$, $f_6(z)=0$, $f_{14}(z)=\dfrac{16}{z^3}$\newline
Output: Class No. 3 in \cite{BEUKERS1989}
\begin{eqnarray*}
0 & = & \left(\frac{d}{dz}\right)^3y+\frac{3}{2}\frac{3z-2}{z(z-1)}\left(\frac{d}{dz}\right)^2y+\frac{3}{112}\frac{116z-35}{z^2(z-1)}\left(\frac{d}{dz}\right)y \\
 & & \ +\frac{195}{2744}\frac{1}{z^2(z-1)}y
\end{eqnarray*}

\subsubsection*{Invariant curves of degree 36 and genus 19}

Input: $F_4$, $F_6$, $F_{14}$, $f_4(z)=8z^{\frac{3}{7}}$, $f_6(z)=3z^{\frac{8}{7}}$, $f_{14}(z)=-4(z^3-1008z^2+9216z-16384)$\newline
Output: Class No. 4.1 in \cite{BEUKERS1989}
\begin{eqnarray*}
0 & = & \left(\frac{d}{dz}\right)^3y+\frac{1}{14}\frac{41z-20}{z(z-1)}\left(\frac{d}{dz}\right)^2y+\frac{1}{196}\frac{173z+16}{z^2(z-1)}\left(\frac{d}{dz}\right)y \\
 & & \ +\frac{9}{2744}\frac{1}{z^2(z-1)}y
\end{eqnarray*}
Input: $F_4$, $F_6$, $F_{14}$, $f_4(z)=-32z^{\frac{3}{7}}$, $f_6(z)=z^{\frac{1}{7}}(5z-128)$, $f_{14}(z)=4(z^3-14624z^2+591872z-16384).$\newline
Output: Class No. 4.2 in \cite{BEUKERS1989}
\begin{eqnarray*}
0 & = & \left(\frac{d}{dz}\right)^3y+\frac{1}{14}\frac{41z-20}{z(z-1)}\left(\frac{d}{dz}\right)^2y+\frac{1}{196}\frac{173z+72}{z^2(z-1)}\left(\frac{d}{dz}\right)y \\
 & & \ +\frac{9}{2744}\frac{1}{z^2(z-1)}y
\end{eqnarray*}

\subsubsection*{Fricke degree $12$ pencil}

Input: $F_4$, $F_6$, $F_{14}$, $f_4(z)=(-\mu)^{-\frac{1}{9}}$, $f_6(z)=(-\mu)^\frac{1}{3}$, $f_{14}(z)=(-\mu)^{-\frac{1}{9}}(z+\frac{88}{3})$\newline
Output: Equation (3.9) in \cite{KATO2004} 
\begin{eqnarray*}
0 & = & \left(\frac{d}{dz}\right)^3y+\left(\frac{3}{2}\frac{\partial}{\partial z}P(\mu,z)-\frac{1}{z-z_4}\right)\left(\frac{d}{dz}\right)^2y+\left(\frac{43}{28}z+\frac{1729\mu^2-3628\mu-640}{21\mu}\right.\\
 & & \ \left.-\frac{(2187\mu^2-15004\mu-3200)(27\mu+4)(\mu-4)}{63\mu^2(z-z_4)}\right)\frac{1}{P(\mu,z)}\left(\frac{d}{dz}\right)y \\
 & & \ +\left(-\frac{15}{14^3}-\frac{5(27\mu+4)(\mu-4)}{196\mu(z-z_4)}\right)\frac{1}{P(\mu,z)}y
\end{eqnarray*}
where
\begin{eqnarray*}
z_4 & = & -\frac{81\mu^2-432\mu-80}{3\mu},\\
P(\mu,z) & = & -\frac{\delta_4(27\mu+4)^2(\mu-4)^2}{27\mu^3},
\end{eqnarray*}
and
\begin{eqnarray*}
\delta_4 & = & 729\mu^2-7120\mu-2000.
\end{eqnarray*}

\bibliographystyle{plain}
\bibliography{DGT}

\end{document}